\documentclass{article}
\usepackage{amssymb,amsmath,amsthm,graphicx,color}

\newcounter{nofigs}

\def\utr{\, \underline{\triangleright}\, }
\def\otr{\, \overline{\triangleright}\, }

\newtheorem{theorem}{Theorem}

\newtheorem{proposition}[theorem]{Proposition}

\theoremstyle{definition}
\newtheorem{example}{Example}
\newtheorem{definition}{Definition}

\date{}

\title{\Large \textbf{Biquandle Brackets and Knotoids}}

\author{Neslihan G\"ug\"umc\"u\footnote{Email:nesli@central.ntua.gr }
 \and
Sam Nelson\footnote{Email: Sam.Nelson@cmc.edu. Partially supported by 
Simons Foundation collaboration grant 316709}\and
Natsumi Oyamaguchi\footnote{Email: natsumi.3-29.math@diary.ocn.ne.jp}}

\begin{document}
\maketitle

\begin{abstract}
\textit{Biquandle brackets} are a type of quantum enhancement of the 
biquandle counting invariant for oriented knots and links, defined by
a set of skein relations with coefficients which are functions of biquandle
colors at a crossing. In this paper we use biquandle brackets to enhance the
biquandle counting matrix invariant defined by the first two authors in
\cite{GN}. We provide examples to illustrate the method of calcuation and to
show that the new invariants are stronger than the previous ones.
\end{abstract}

\parbox{5.5in} {\textsc{Keywords:} Knotoids, biquandles, enhancements of 
counting invariants, biquandle brackets, quantum enhancements

\section*{Acknowledgement}
The first author thanks cordially to Oberwolfach Institute for Mathematics for the peaceful environment provided for completing her contribution to the paper while her stay as a Leibniz Fellow. 
\smallskip

\textsc{2010 MSC:} 57M27, 57M25}

\section{\large\textbf{Introduction}}\label{I}

Introduced in \cite{T}, \textit{knotoids} are a generalization of tangles 
in which the endpoints can lie any region of the planar complement of the
tangle and are not permitted to move over or under other strands, instead 
remaining confined to the region of the plane in which they start. Knotoids
have been the subject of much recent study; see for example \cite{GK1, GL1, AHKS, GGLKS}.

In \cite{FRS}, algebraic structures known as \textit{biquandles} were 
introduced and used in later work to define invariants of oriented knots
and links via coloring; see \cite{EN} for more details. 
In \cite{GN}, the first two authors introduced biquandle colorings of knotoids.
In particular, the regional confinement of knotoid endpoints enables the
arrangement of biquandle colorings numbers into a matrix-valued invariant
which is stronger than the counting invariant alone. More precisely, the sum
of the entries in the matrix gives the total number of colorings, but 
the distribution of coloring numbers within the matrix is also invariant
under Reidemeister moves and can distinguish knotoids with equal numbers of 
colorings.

In \cite{NOR}, the second author and collaborators introduced 
\textit{biquandle brackets}, skein invariants for biquandle-colored knots 
and links, and used them to define an infinite family of oriented link 
invariants which include the classical quantum invariants such as the 
Alexander-Conway, Jones, HOMFLYPT and Kauffman polynomials on the one hand 
and the quandle and biquandle 2-cocycle invariants defined in \cite{CJKLS}
and other recent work on the other hand as special cases. In \cite{NO}, the 
second two authors defined a graphical calculus known as
\textit{trace diagrams} for computing biquandle brackets recursively as opposed
to using the state-sum approach. For an overview of biquandle brackets, see
\cite{N}.

In this paper we extend biquandle brackets to the case of knotoids, in 
particular enhancing the biquandle coloring matrix with biquandle brackets
to obtain a matrix-vlaued biquandle bracket invariant of knotoids. The paper is
organized as follows. In Section \ref{K} we review the basics of knotoid theory.
In Sections \ref{Bi} and  \ref{B} we review the basics of biquandles and 
biquandle brackets. In Section \ref{BBK} we introduce the new invariants and 
give some computational
examples, in particular demonstrating that the new invariant is a proper
generalization of the previous invariants. In Section \ref{Q} we conclude with
some open questions for future research.

\section{\large{\textbf{Knotoids}}}\label{K}

A \textit{knotoid diagram} $K$ is a generic immersion of $[0,1]$ into an orientable surface $\Sigma$ with a finite number of double points that are transversal and endowed with over/under information and called \textit{crossings} of $K$.  The images of $0$ and $1$ are regarded as the endpoints of $K$, and are called the \textit{tail} and the \textit{head} of $K$. The endpoints of $K$ are distinct from each other and from any of the crossings of $K$. A knotoid is always oriented from tail to head. 
\[\includegraphics{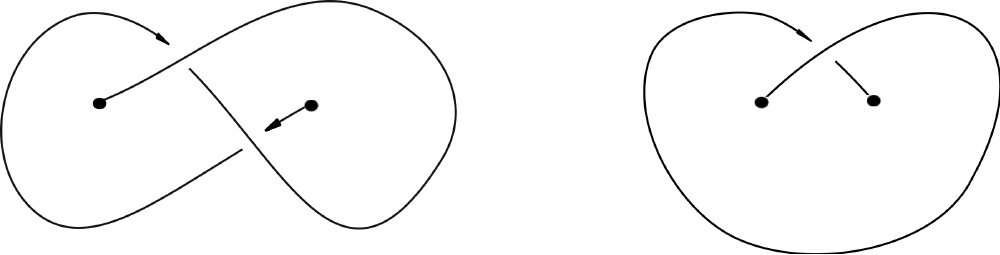}\]

Two knotoid diagrams in the surface $\Sigma$ are considered to be \textit{equivalent} if they are related to each other by a finite sequence of Reidemeister moves which take place in local disks free of endpoints and the isotopy of the surface $\Sigma$.  A \textit{knotoid} in $\Sigma$ is then considered to be an equivalence class of equivalent knotoid diagrams in $\Sigma$.  Specifically, a knotoid in $S^2$ is called a \textit{spherical knotoid} and a knotoid in $\mathbb{R}^2$ is called a \textit{planar knotoid}. 

One of the features that makes knotoids interesting is that considering them in $S^2$ and $\mathbb{R}^2$  yields two different theories unlike knots \cite{T}.  This is due to the Whitney flip moves enabled in $S^2$; the classification of knotoids gets coarser in $S^2$ by this type of moves. The reader can enjoy verifying that the two knotoids depicted above 
 are two nontrivial knotoids in $\mathbb{R}^2$ but happens to be equivalent and trivial in $S^2$. 

The connected sum operation on classical knots extends to knotoids \cite{T}. Two knotoids $K_1, K_2$ in $\Sigma_1, \Sigma_2$  can be summed up by tying the head of $K_1$ to the tail of $K_2$ through an orientation reversing homeomorphism between $\Sigma_1-D_1$ and $\Sigma_2-D_2$, where $D_1, D_2$ are sufficiently small neighborhoods containing the head of $K_1$ and the tail of $K_2$, respectively.  In particular, when the $\Sigma_1=\Sigma_2=S^2$, the set of knotoids in $S^2$ carry a monoid structure with the connected sum operation. We call a knotoid in $S^2$ \textit{prime}  if it is the connected sum of only itself and the trivial knotoid. A \textit{composite} knotoid is a connected sum  of a finite number of non-trivial prime knotoids. 

Knotoids in $S^2$ and $\mathbb{R}^2$ can be regarded as $\theta$-graphs and open-ended smooth curves embedded in $\mathbb{R}^3$, respectively. Turaev showed that there is a bijection between the set of knotoids in $S^2$ and the set of isotopy classes of $\theta$-graphs \cite{T}. In \cite{T} it is shown that every composite knotoid has a unique prime knotoid decomposition up to knotoid equivalence through this interpretation \cite{Ma}. The first author and Kauffman showed a $1-1$ correspondence between knotoids lying in a plane and line isotopy classes of open ended smooth curves embedded in $\mathbb{R}^3$ \cite{GK1}. During the line isotopy, the endpoints of the curve are assumed to remain on the two lines passing through the endpoints and are orthogonal to plane of the knotoids and the rest of the curve is subject to the ambient isotopy of the complementary space to the lines.  Application of knotoids for  a topological classification of proteins and polymers  \cite{GGLKS, GS} is realized through this interpretation.

\section{\large\textbf{Biquandles}}\label{Bi}
We begin with a definition (see \cite{EN} for more).

\begin{definition}
Let $X$ be a set. A \textit{biquandle structure} on $X$ is a
pair of maps $\utr,\otr:X\times X\to X$ satisfying
\begin{itemize}
\item[(i)] For all $x\in X$, $x\utr y=x\otr y$,
\item[(ii)] The maps $\alpha_y,\beta_y:X\to X$for all $y\in X$ 
and $S:X\times X\to X\times X$ defined by
\[\alpha_y(x)=x\otr y,\quad \beta_y(x)=x\utr y\quad \mathrm{and}\quad
S(x,y)=(y\otr x, x\utr y)\]
are bijective, and
\item[(iii)] For all $x,y,z\in X$ we have the \textit{exhcange laws}:
\[
\begin{array}{rcl}
(x\utr y)\utr (z\utr y)& = & (x\utr z)\utr(y\otr z) \\
(x\otr y)\utr (z\otr y)& = & (x\utr z)\otr(y\utr z) \\
(x\otr y)\otr (z\otr y)& = & (x\otr z)\otr(y\utr z).
\end{array}
\] 
\end{itemize}
We note that axiom (ii) is equivalent to the \textit{adjacent labels rule},
which says that in the ordered quadruple $(x,y,x\utr y,y\otr x)$, any two 
neighboring entries (including $(y\otr x, x)$ determine the other two.
A \textit{biquandle} is a set $X$ with a choice of biquandle structure.
\end{definition}

\begin{example}
Any $\mathbb{Z}[t^{\pm 1,s^{\pm 1}}]$-module has a biquandle structure known as
an \textit{Alexander biquandle} defined by
\[x\utr y = tx+(s-t)y \quad\mathrm{and}\quad  x\otr y = sx.\]
In particular, a choice of units $t,s\in\mathbb{Z}_n$ defines an
Alexander biquandle structure on $\mathbb{Z}_n$.
\end{example}

\begin{example}
We can define biquandle structures on a finite set $X=\{1,2,3,\dots, n\}$
by specifying the operation tables of $\utr$ and $\otr$. For example,
thinking of $\mathbb{Z}_5=\{1,2,3,4,5\}$ (with the class of zero represented
by 5 so our row/column numbers can start at 1), the Alexander biquandle
structure on $X$ determined by $s=2$ and $t=3$, i.e.,
\[x\utr y = 3x+4y, \quad x\otr y =2\]
is expressed by the operation tables
\[
\begin{array}{r|rrrrr} 
\utr & 1 & 2 & 3 & 4 & 5 \\\hline
1 &    2 & 1 & 5 & 4 & 3 \\
2 &    5 & 4 & 3 & 2 & 1 \\
3 &    3 & 2 & 1 & 5 & 4 \\
4 &    1 & 5 & 4 & 3 & 2 \\
5 &    4 & 3 & 2 & 1 & 5 \\
\end{array}\quad
\begin{array}{r|rrrrr}
\otr & 1 & 2 & 3 & 4 & 5 \\\hline
1 & 2 & 2 & 2 & 2 & 2 \\ 
2 & 4 & 4 & 4 & 4 & 4 \\
3 & 1 & 1 & 1 & 1 & 1 \\
4 & 3 & 3 & 3 & 3 & 3 \\ 
5 & 5 & 5 & 5 & 5 & 5.
\end{array}
\]
\end{example}

\begin{definition}
Let $X$ be a biquandle and and let $K$ be a knotoid diagram in $S^2$. A 
\textit{biquandle coloring} of $K$ is an assignment of elements of
$X$ to each semiarc in $K$ such that the following
conditions are satisfied at every crossing:
\[\includegraphics{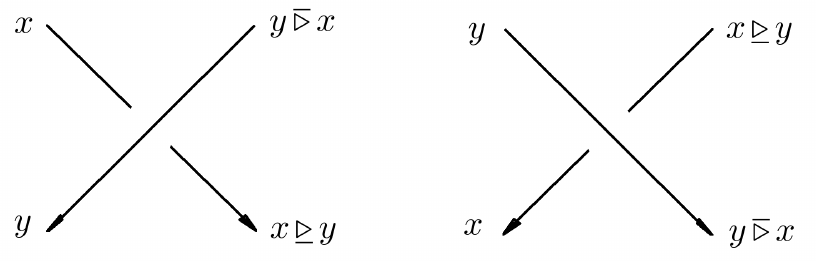}\]
\end{definition}

The biquandle axioms are chosen such that given a biquandle coloring of
one side of a Reidemeister move, there is a unique biquandle coloring of the
other side of the move which agrees on the boundary of the neighborhood of
the move. It follows that the number of biquandle colorings is a knotoid 
invariant, computable from any diagram $K$ of our knotoid. 
Moreover, in \cite{GN}, the first two listed authors observed that there 
is no Reidemeister move which can change the colors of the head and tail 
semiarcs; it then follows that that the \textit{biquandle coloring matrix}, 
whose entry in row $j$ column $k$ (where $X=\{1,2,\dots, n\}$) is the number 
of colorings with tail semiarc colored $j$ and head colored $k$, is an
invariant of knotoids.

\begin{example}
Let $X=\{1,2,3\}$ and consider the biquandle structure on $X$ given by the
operation tables
\[
\begin{array}{r|rrr} 
\utr & 1 & 2 & 3 \\\hline
1 &    2 & 1 & 3 \\
2 &    1 & 3 & 2 \\
3 &    3 & 2 & 1 \\
\end{array}\quad
\begin{array}{r|rrr}
\otr & 1 & 2 & 3 \\\hline
1 & 2 & 2 & 2 \\
2 & 3 & 3 & 3 \\
3 & 1 & 1 & 1.
\end{array}
\]
\[\includegraphics{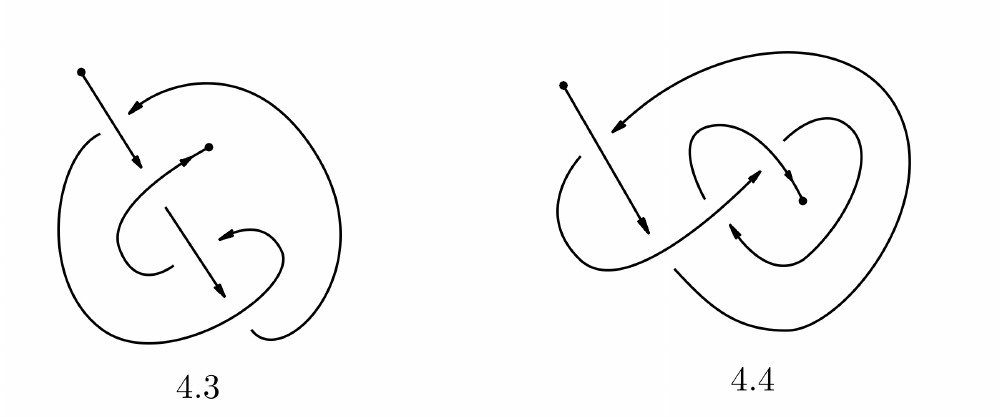}\]
The unknotoid has three $X$-colorings. 
The knotoid numbered $4.4$ in \cite{BA} has no $X$-colorings, distinguishing it from the
unknotoid. The knotoid $4.3$ has three $X$-colorings
\[\includegraphics{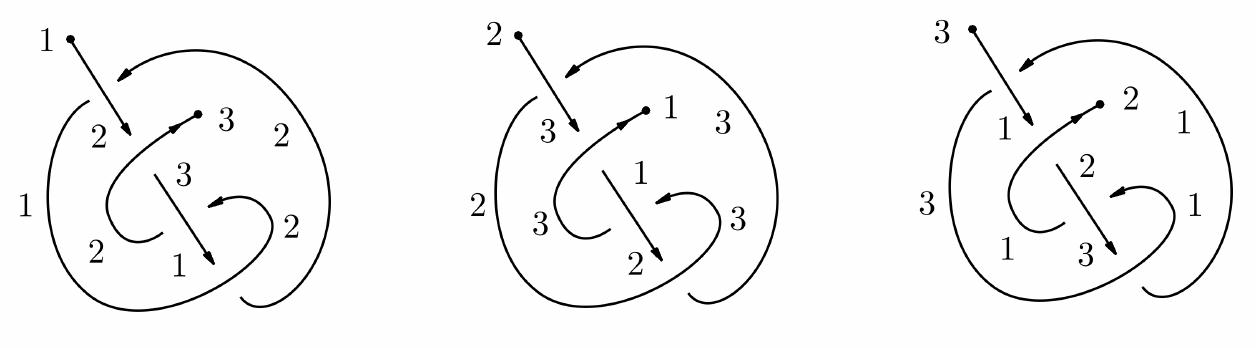}\] 
but its coloring matrix value 
\[\left[\begin{array}{rrr}
0 & 0 & 1\\ 
1 & 0 & 0 \\
0 & 1 & 0
\end{array}\right]\ne
\left[\begin{array}{rrr}
1 & 0 & 0 \\ 
0 & 1 & 0 \\
0 & 0 & 1\\ 
\end{array}\right]\] distinguishes it from the unknotoid, which has
coloring matrix given by the $3\times 3$ identity matrix.
\end{example}

One justification for arranging the coloring numbers as a matrix
is given by the following result:
\begin{proposition}
Let $L_1 \# L_2$ denote the connected sum of the knotoids $L_1$ and $L_2$ 
in $S^2$.  Then for any biquandle
$X$, the coloring matrix of $L_1 \# L_2$ is the matrix product of the
coloring matrices of $L_1$ and $L_2$. More precisely, if we denote the
coloring matrix of $L$ as $M_X(L)$ then we have
\[M_{L_1\#L_2}=M_{L_1}M_{L_2}.\]

\end{proposition}

\begin{proof}
Each $X$-coloring of $L_1\#L_2$ has some biquandle element $k\in X$
on the semiarc which splits to become the head of $L_1$ and the tail
of $L_2$. Thus, the number of biquandle colorings of $L_1\#L_2$ starting with
color $j$ and ending with color $l$ is the sum over $k\in X$ of the number
of colorings of $L_1$ starting with $j$ and ending with $k$ times the
number of colorings of $L_2$ starting with $k$ and ending with $l$,
i.e., the dot product of the $j$th row of 
$M_{L_1}$ with the $l$th column of $M_{L_2}$.

\end{proof}

\section{Biquandle Brackets}\label{B}
Next we review biquandle brackets. We begin with a definition from
\cite{NOR}.

\begin{definition}
Let $X$ be a biquandle and $R$ a commutative ring with identity. A 
\textit{biquandle bracket structure} or \textit{$X$-bracket} on $R$ consists 
of maps $A,B:X\times X\to R^{\times}$ assigning units $A_{x,y}, B_{x,y}$ in $R$
to each ordered pair $(x,y)$ of elements of $X$ subject to the following
conditions:
\begin{itemize}
\item[(i)] For all $x\in X$, the elements $-A_{x,x}^2B_{x,x}^{-1}$ are equal,
with their common value denoted by $w$,
\item[(ii)] For all $x,y\in X$, the elements 
$-A_{x,y}B_{x,y}^{-1}-A_{x,y}^{-1}B_{x,y}$ are equal, with their common value
denoted by $\delta$, and
\item[(iii)] For all $x,y,z\in X$, the five equations
\[\begin{array}{rcl}
A_{x,y}A_{y,z}A_{x\utr y,z\otr y} & = & A_{x,z}A_{y\otr x,z\otr x}A_{x\utr z,y\utr z} \\
A_{x,y}B_{y,z}B_{x\utr y,z\otr y} & = & B_{x,z}B_{y\otr x,z\otr x}A_{x\utr z,y\utr z} \\
B_{x,y}A_{y,z}B_{x\utr y,z\otr y} & = & B_{x,z}A_{y\otr x,z\otr x}B_{x\utr z,y\utr z} \\
A_{x,y}A_{y,z}B_{x\utr y,z\otr y} & = & 
A_{x,z}B_{y\otr x,z\otr x}A_{x\utr z,y\utr z} 
+A_{x,z}A_{y\otr x,z\otr x}B_{x\utr z,y\utr z} \\ 
& & +\delta A_{x,z}B_{y\otr x,z\otr x}B_{x\utr z,y\utr z} 
+B_{x,z}B_{y\otr x,z\otr x}B_{x\utr z,y\utr z} \\
B_{x,y}A_{y,z}A_{x\utr y,z\otr y} 
+A_{x,y}B_{y,z}A_{x\utr y,z\otr y} & & \\
+\delta B_{x,y}B_{y,z}A_{x\utr y,z\otr y} 
+B_{x,y}B_{y,z}B_{x\utr y,z\otr y}  
& = & B_{x,z}A_{y\otr x,z\otr x}A_{x\utr z,y\utr z}. \\
\end{array}\]
are satisfied.
\end{itemize}
\end{definition}

Given a finite biquandle structure on $X=\{1,2,\dots, n\}$ and a ring $R$, 
we can specify a biquandle bracket structure $\beta$ with a block matrix 
$[A|B]$ whose entries tell us $A_{x,y}$ and $B_{x,y}$ for $x,y\in X$.

A biquandle bracket defines an $R$-valued skein invariant of $X$-colored
oriented knots and links under the following skein relations.
\[\includegraphics{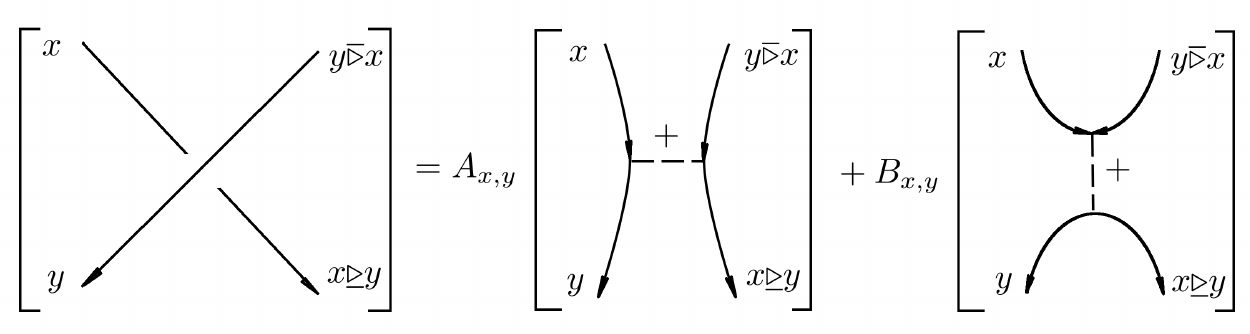}\]
\[\includegraphics{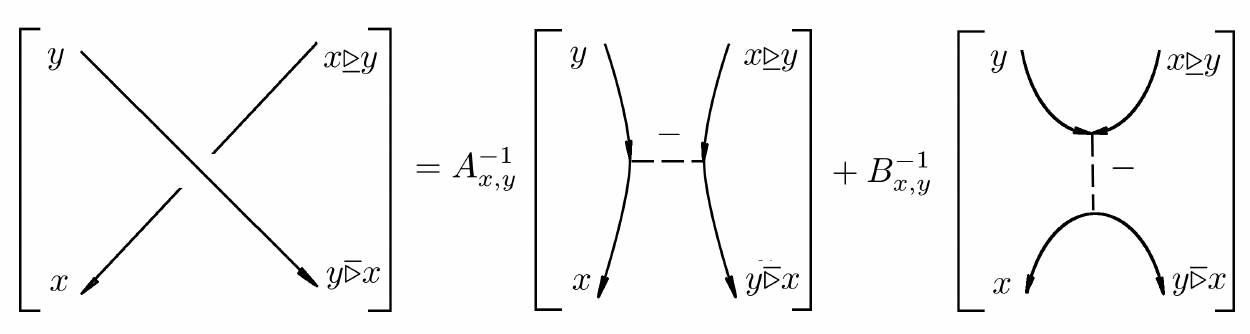}\]
Applying these skein relations at every crossing, we obtain an $R$-linear 
combination of \textit{trace diagrams}, trivalent graphs with certain 
distinguished edges called \textit{traces} marked with signs according to 
crossing types and indicating
the sites of smoothing. Replacing each trace diagram with $\delta^cw^{n-p}$
where $c$ is the number of circular components (called \textit{Kauffman 
states}) after deleting the traces and 
$n-p$ is the number of negative signed traces minus the number of
positive signed traces yields an invariant of $X$-colored Reidemeister moves,
denoted $\beta(L_c)$ where $L_c$ is a biquandle-colored knot or link.
Then the \textit{biquandle bracket polynomial} of an oriented knot $L$
is the sum over the set $\mathcal{C}(L,X)$ of $X$-colorings of $L$ of 
contributions $u^{\beta(L_c)}$, i.e.,
\[\Phi_X^{\beta}(L)=\sum_{L_c\in \mathcal{C}(L,X)} u^{\beta(L_c)}.\]
This polynomial is an invariant of oriented knots and links for each 
biquandle $X$ and biquandle bracket $\beta$ over each commutative ring $R$; 
this infinite family of oriented link invariants includes the classical quantum
invariants (Alexander-Conway, Jones, HOMFLYPT, Kauffman polynomials) and
biquandle $2$-cocycle invariants as special cases, but also includes
other invariants. It is perhaps worth noting that the original version of 
biquandle brackets defined in \cite{NOR} did not use traces, but only the 
state-sum formulation in which all smoothings are done at once; trace diagrams 
were introduced in \cite{NO} to allow for recursive computation of biquandle
brackets, subject to some restrictions on moving strands past traces.
See \cite{NO} for more.

\section{\large\textbf{Biquandle Brackets and Knotoids}}\label{BBK}

To generalize biquandle brackets from knots to knotoids, there are two
important points. First, the set of smoothed states after deleting traces 
now includes open-ended components as well as closed loops. The simplest 
option is to treat these open-ended component sthe same as the loop components,
i.e. assign it a value of $\delta$ as well. Secondly, instead of
simply summing the contributions of $u^{\beta(K_f)}$ over the set of $X$-colorings
$K_f$ of our knotoid $K$, we will sum these contributions as entries
in the biquandle coloring matrix of the knotoid. More precisely, we have:

\begin{definition}
Let $K$ be an oriented knotoid, $X=\{1,2,\dots, n\}$ a finite biquandle, 
$R$ a commutative ring
with identity and $\beta$ an $X$-bracket over $R$. We define the 
\textit{biquandle bracket matrix} of $K$ with respect to $\beta$ to be the
matrix 
\[\Phi_X^{\beta}(K)=\left[\begin{array}{rrr} 
\beta_{11} & \dots & \beta_{1n} \\
\vdots & \ddots &\vdots \\
\beta_{n1} & \dots & \beta_{nn} \\
\end{array}\right]\]
where
\[\beta_{jk}=\sum_{f\in H_{jk}} u^{\beta(f)}\]
and $H_{jk}$ is the set of biquandle colorings of $K$ with tail color $j$
and head color $k$.
\end{definition}

By construction, we have our main result:
\begin{proposition}
$\Phi_X^{\beta}(K)$ is an invariant of knotoids.
\end{proposition}

\begin{example}
Let us illustrate the process of computation of the invariant. Let 
$X=\{1,2,3\}$ be the biquandle with operation tables
\[
\begin{array}{r|rrr}
\utr & 1 & 2 & 3 \\ \hline
1 & 2 & 1 & 3 \\
2 & 1 & 3 & 2 \\
3 & 3 & 2 & 1
\end{array}
\quad\mathrm{and}\quad
\begin{array}{r|rrr}
\otr & 1 & 2 & 3 \\ \hline
1 & 2 & 2 & 2 \\
2 & 3 & 3 & 3 \\
3 & 1 & 1 & 1 
\end{array}.
\]
Then $X$ has biquandle bracket values with $R=\mathbb{Z}_5$ coefficients
including
\[\beta=
\left[\begin{array}{rrr|rrr}
1 & 2 & 4 & 4 & 3 & 1 \\
1 & 1 & 2 & 4 & 4 & 3 \\
4 & 4 & 1 & 1 & 1 & 4
\end{array}\right].\]
This data encodes $3^2=9$ pairs of skein relations at positive and negative 
crossings with different biquandle colorings; for example, the 
$(1,2)$ positions $A_{1,2}=2$ and $B_{1,2}=3$ 
(so $A_{1,2}^{-1}=3$ and $B_{1,2}^{-1}=2$) specify the skein relations
\[\includegraphics{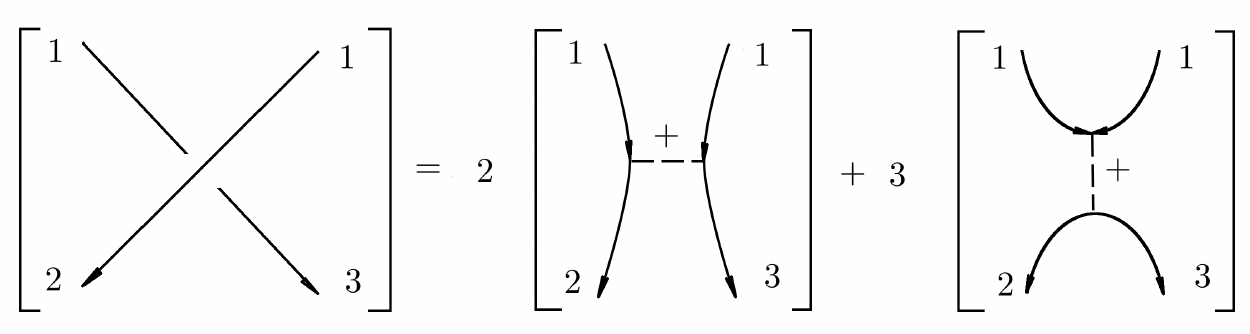}\]
and
\[\includegraphics{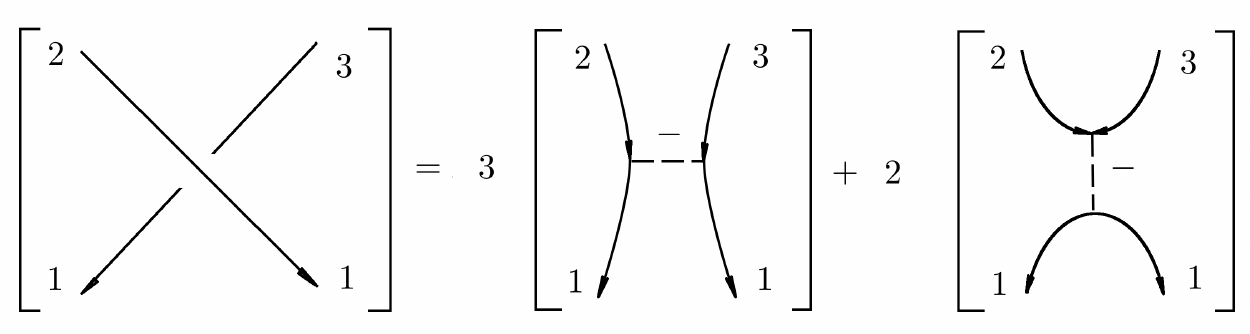}.\]
We also have $\delta=-A_{11}^{-1}B_{11}-A_{11}B_{11}^{-1}=-1(4)-1(4)=2$ and 
$w=-A_{11}^2B_{11}^{-1}=-1^2(4)=1$.

The knotoid $3.1$ has three colorings by $X$, the same number as the unknotoid.
\[\includegraphics{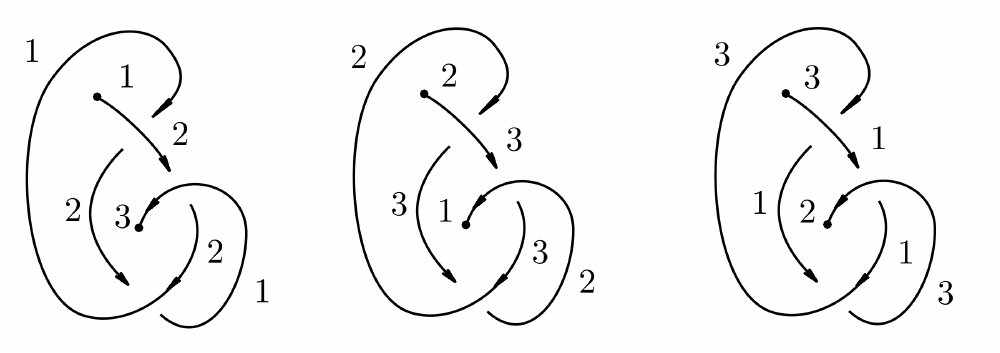}\]
Taking the first coloring, let us compute its $\beta$-value. There are $2^3=8$
smoothed states, each contributing its product of smoothing coefficients, power of $\delta$ and power of $w$ to the $\beta$-value.
\[\includegraphics{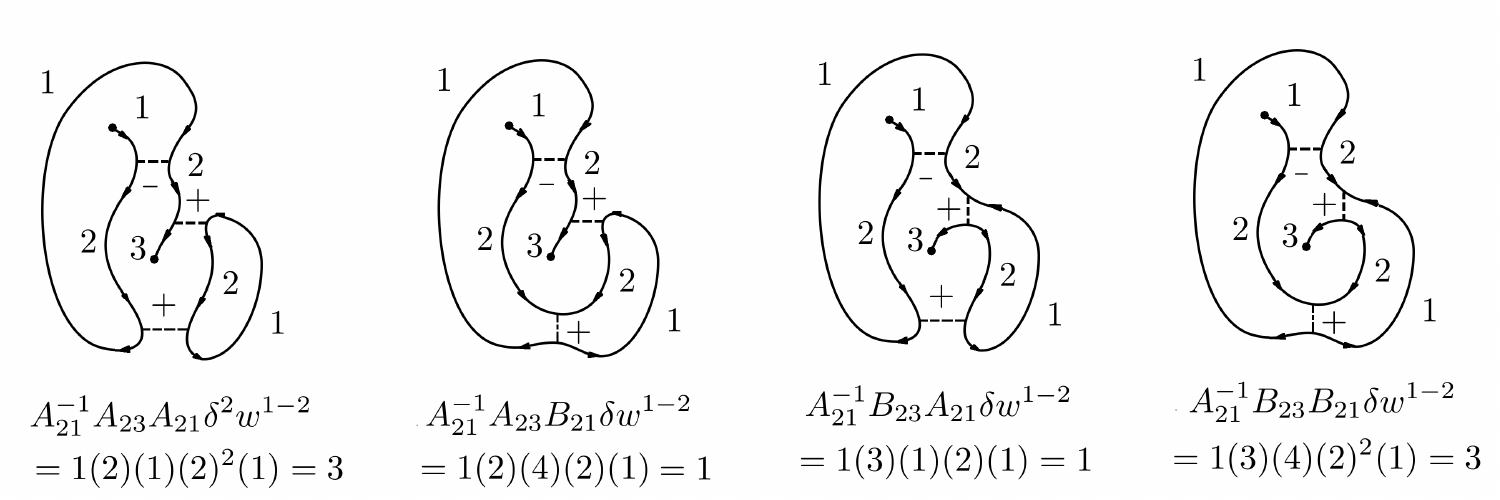}\]
\[\includegraphics{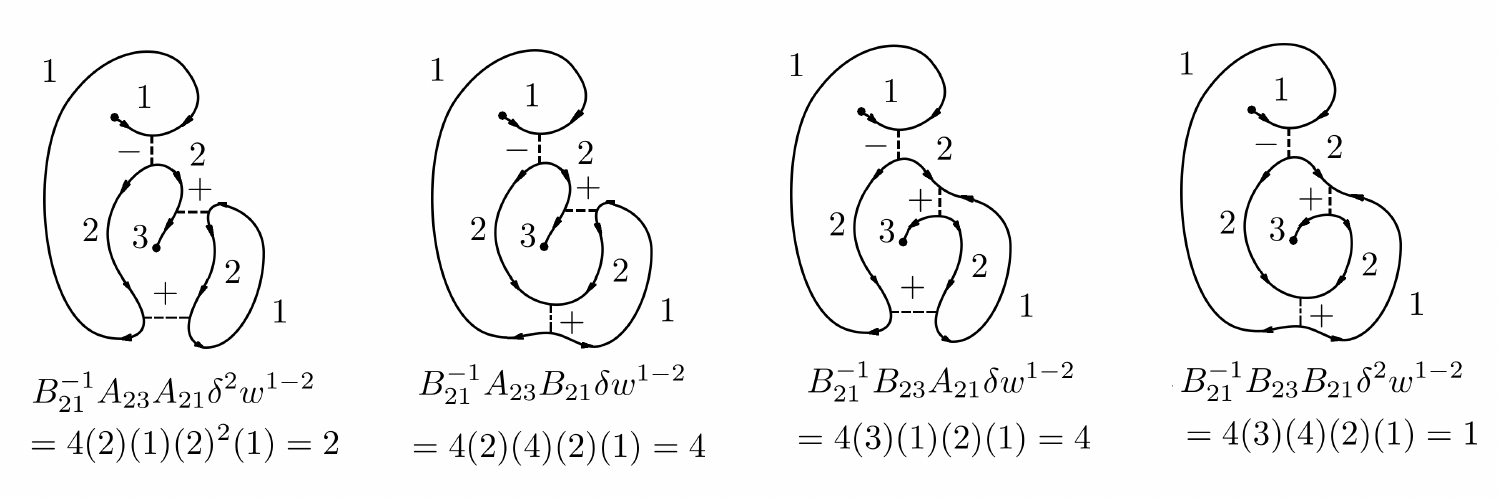}\]
Then this coloring has $\beta$-value $3+1+1+3+2+4+4+1=4$ and contributes $u^4$
to the entry in row 1 column 3 of the invariant matrix. Repeating for the other
colorings, we obtain invariant value
\[\Phi_X^{\beta}(3.1)=\left[\begin{array}{rrr}
0 & 0 & u^4 \\
u^3 & 0 & 0 \\
0 & u^4 & 0 \\
\end{array}\right].\]
We note that from this matrix we can obtain the counting matrix by 
specializing $u=1$ and the biquandle bracket polynomial by summing the entries
of the matrix.
\end{example}

\begin{example}
Let $X$ be the biquandle structure on $\{1,2,3\}$ given by the operation tables
\[\begin{array}{r|rrr}
\utr & 1 & 2 & 3\\ \hline
1 & 2 & 3 & 1 \\
2 & 3 & 1 & 2 \\
3 & 1 & 2 & 3
\end{array}
\quad 
\begin{array}{r|rrr}
\otr & 1 & 2 & 3\\ \hline
1 & 2 & 2 & 2 \\
2 & 1 & 1 & 1 \\
3 & 3 & 3 & 3
\end{array}
\]
and let $\beta$ be the $X$-bracket over $\mathbb{Z}_7$ given by
\[
\left[\begin{array}{rrr|rrr}
1 & 2 & 2 & 3 & 6 & 6\\
3 & 1 & 4 & 2 & 3 & 5 \\
1 & 6 & 1 & 3 & 4 & 3
\end{array}\right].
\]
Then the knotoids in the table at \cite{BA} have the following
biquandle bracket matrix values.
\[
\begin{array}{r|l|r|l}\Phi_X^{\beta}(K) & K &\Phi_X^{\beta}(K) & K \\ \hline
& & & \\
\left[\begin{array}{ccc}1 & 0 & 0 \\0 & 1 & 0 \\0 & 0 & 1\end{array}\right] &
\begin{array}{l}
2.1, 4.4, 4.5, 5.5, 5.10, 5.11, \\
5.12, 5.13, 5.15, 5.26 \end{array}  & 
\left[\begin{array}{ccc}0 & 0 & 0 \\0 & 0 & 0 \\0 & 0 & 3\end{array}\right] & 
5.14, 5.16 \\ & \\
\left[\begin{array}{ccc}u & 0 & 0 \\0 & u & 0 \\0 & 0 & u\end{array}\right] & 5.9 & 
\left[\begin{array}{ccc}u & 0 & 0 \\0 & u^2 & 0 \\0 & 0 & u^4\end{array}\right] & 5.27\\ & \\
\left[\begin{array}{ccc}u & 0 & 0 \\0 & u^4 & 0 \\0 & 0 & u^2\end{array}\right] & 3.1,4.3 & 
\left[\begin{array}{ccc}u^2 & 0 & 0 \\0 & u^2 & 0 \\0 & 0 & u^2\end{array}\right] & 5.6, 5.21\\ & \\
\left[\begin{array}{ccc}u^2 & 0 & 0 \\0 & u^4 & 0 \\0 & 0 & u\end{array}\right] & 5.22 & 
\left[\begin{array}{ccc}u^3 & 0 & 0 \\0 & u^3 & 0 \\0 & 0 & u^3\end{array}\right] & 5.18\\ & \\
\end{array}\]\[\begin{array}{r|l|r|l}\Phi_X^{\beta}(K) & K & \Phi_X^{\beta}(K) & K \\ \hline  & \\
\left[\begin{array}{ccc}u^3 & 0 & 0 \\0 & u^6 & 0 \\0 & 0 & u^5\end{array}\right] & 5.3, 5.20 & 
\left[\begin{array}{ccc}u^4 & 0 & 0 \\0 & u & 0 \\0 & 0 & u^2\end{array}\right] & 5.19\\ & \\
\left[\begin{array}{ccc}u^4 & 0 & 0 \\0 & u^2 & 0 \\0 & 0 & u\end{array}\right] & 5.28  & 
\left[\begin{array}{ccc}u^4 & 0 & 0 \\0 & u^4 & 0 \\0 & 0 & u^4\end{array}\right] & 5.30\\ & \\
\left[\begin{array}{ccc}u^5 & 0 & 0 \\0 & u^5 & 0 \\0 & 0 & u^5\end{array}\right] & 4.8  & 
\left[\begin{array}{ccc}u^6 & 0 & 0 \\0 & u^5 & 0 \\0 & 0 & u^3\end{array}\right] & 4.9, 5.29\\ & \\
\left[\begin{array}{ccc}0 & 0 & 0 \\0 & 0 & 0 \\0 & 0 & u+u^2+u^4\end{array}\right] & 4.1, 4.2, 5.7, 5.8  & 
\left[\begin{array}{ccc}0 & 0 & 0 \\0 & 0 & 0 \\0 & 0 & u^3+u^5+u^6\end{array}\right] & 5.23\\  & \\
\left[\begin{array}{ccc}3 & 0 & 0 \\0 & 3 & 0 \\0 & 0 & 3\end{array}\right] & 5.1, 5.2  & 
\left[\begin{array}{ccc}3u & 0 & 0 \\0 & 3u & 0 \\0 & 0 & 3u\end{array}\right] & 4.7\\  & \\
\left[\begin{array}{ccc}3u^2 & 0 & 0 \\0 & 3u^4 & 0 \\0 & 0 & 3u\end{array}\right] & 4.6  & 
\left[\begin{array}{ccc}3u^3 & 0 & 0 \\0 & 3u^6 & 0 \\0 & 0 & 3u^5\end{array}\right] & 5.25\\  & \\
\left[\begin{array}{ccc}3u^5 & 0 & 0 \\0 & 3u^6 & 0 \\0 & 0 & 3u^3\end{array}\right] & 5.24  & 
\left[\begin{array}{ccc}3u^6 & 0 & 0 \\0 & 3u^6 & 0 \\0 & 0 & 3u^6\end{array}\right] & 5.17\\ & \\
\end{array}\]
This example shows that $\Phi_X^{\beta}$ is a stronger invariant than either
the coloring matrix or the biquandle bracket polynomial alone, both of which
in turn are stronger than the biquandle counting invariant. Specifically, the 
knotoids 3.1 and 5.27 have the same counting invariant value 
$\Phi_X^{\mathbb{Z}}=3$ and biquandle counting matrix value (the $3\times 3$ 
identify matrix) as the unknotoid, and both have the same (nontrivial)
biquandle bracket polynomial value $u+u^2+u^4$, but they are
distinguished by their bracket matrices. Many other similar examples can be 
found in the tables.
\end{example}

\section{\large\textbf{Questions}}\label{Q}

We end with some questions for future research. 

In \cite{T}, generalizations of the Jones polynomial for knotoids in $S^2$ and $\mathbb{R}^2$ are 
given by considering the intersection numbers of circular state components with components obtained by closing the open-ended components and if the open-ended state components are nested by circular components or not. What happens when we apply 
this approach to the bases of knotoid biquandle brackets?

In \cite{GN} a longitude is used to enhance the biquandle coloring matrix
for knotoids. What happens when we combine this information with the biquandle
bracket information?

\bibliography{ng-sn-no}{}

\begin{thebibliography}{10}

\bibitem{AHKS}
C.~Adams, A.~Henrich, K.~Kearney, and N.~Scoville.
\newblock Knots related by knotoids.
\newblock {\em Amer. Math. Monthly}, 126(6):483--490, 2019.

\bibitem{BA}
A.~Bartholomew.
\newblock Andrew bartholomew's mathematics page: Knotoids.
\newblock {\em http://www.layer8.co.uk/maths/knotoids/index.htm}, 2015.

\bibitem{CJKLS}
J.~S. Carter, D.~Jelsovsky, S.~Kamada, L.~Langford, and M.~Saito.
\newblock Quandle cohomology and state-sum invariants of knotted curves and
  surfaces.
\newblock {\em Trans. Amer. Math. Soc.}, 355(10):3947--3989, 2003.

\bibitem{EN}
M.~Elhamdadi and S.~Nelson.
\newblock {\em Quandles---an introduction to the algebra of knots}, volume~74
  of {\em Student Mathematical Library}.
\newblock American Mathematical Society, Providence, RI, 2015.

\bibitem{FRS}
R.~Fenn, C.~Rourke, and B.~Sanderson.
\newblock Trunks and classifying spaces.
\newblock {\em Appl. Categ. Structures}, 3(4):321--356, 1995.

\bibitem{GS}
D.~Goundaroulis, J.~Dorier, F.~Benedetti, and A.~Stasiak.
\newblock Studies of global and local entanglements of individual protein
  chains using the concept of knotoids.
\newblock {\em Scientific Reports}, 7, 2017.

\bibitem{GGLKS}
D.~Goundaroulis, N.~G\"ug\"umc\"u, S.~Lambropoulou, J.~Dorier, A.~Stasiak, and
  L.~Kauffman.
\newblock Topological models for open knotted protein chains using the concepts
  of knotoids and bonded knotoids.
\newblock {\em Polymers, Special issue on Knotted and Catenated Polymers}, 9,
  2017.

\bibitem{GK1}
N.~G\"ug\"umc\"u and L.~H. Kauffman.
\newblock New invariants of knotoids.
\newblock {\em European J. Combin.}, 65:186--229, 2017.

\bibitem{GL1}
N.~G\"ug\"umc\"u and S.~Lambropoulou.
\newblock Knotoids, braidoids and applications.
\newblock In {\em Symmetry}, volume 9(12) of {\em Special Issue: Knot Theory
  and Its Applications}. MDPI, 2017.

\bibitem{GN}
N.~G\"{u}g\"{u}mc\"{u} and S.~Nelson.
\newblock Biquandle coloring invariants of knotoids.
\newblock {\em J. Knot Theory Ramifications}, 28(4):1950029, 18, 2019.

\bibitem{Ma}
T.~Motohashi.
\newblock Prime decompositions of a {$\theta_n$}-curve in {$S^3$}.
\newblock {\em Topology Appl.}, 93(2):161--172, 1999.

\bibitem{N}
S.~Nelson.
\newblock A survey of quantum enhancements.
\newblock {\em arXiv:1805.12230}, 2018.

\bibitem{NOR}
S.~Nelson, M.~E. Orrison, and V.~Rivera.
\newblock Quantum enhancements and biquandle brackets.
\newblock {\em J. Knot Theory Ramifications}, 26(5):1750034, 24, 2017.

\bibitem{NO}
S.~{Nelson} and N.~{Oyamaguchi}.
\newblock {Trace diagrams and biquandle brackets.}
\newblock {\em {Int. J. Math.}}, 28(14):24, 2017.

\bibitem{T}
V.~Turaev.
\newblock Knotoids.
\newblock {\em Osaka J. Math.}, 49(1):195--223, 2012.

\end{thebibliography}
\bibliographystyle{abbrv}

\bigskip
\noindent
\textsc{Department of Mathematics\\
Izmir Institute of Technology\\
G\"ulbahce Mah. 35430 \\
Urla, Izmir Turkey}
\bigskip

\noindent
\textsc{Department of Mathematical Sciences \\
Claremont McKenna College \\
850 Columbia Ave. \\
Claremont, CA 91711} 

\bigskip

\noindent
\textsc{Department of Teacher Education \\
Shumei University \\
1-1 Daigaku-cho, Yachiyo \\
Chiba Prefecture 276-0003, Japan}

\end{document}